\documentclass[11pt,a4paper]{article}

\usepackage{datetime}

\usepackage[T1,T5]{fontenc}

\usepackage[all]{xy}

\usepackage{amssymb, amsmath, amsthm}

\usepackage{mathptmx}
\usepackage[scaled=.90]{helvet}
\usepackage{courier}

\usepackage[pdftex]{hyperref}

\hypersetup{
  colorlinks = true,
  linkcolor = black,
  urlcolor = blue,
  citecolor = black,
}

\renewcommand{\phi}{\varphi}

\newcommand{\bbF}{\mathbb{F}}

\newcommand{\rmA}{\mathrm{A}}

\newcommand{\rmH}{\mathrm{H}}

\newcommand{\rmM}{\mathrm{M}}
\newcommand{\rmP}{\mathrm{P}}

\newcommand{\rmS}{\mathrm{S}}

\DeclareMathOperator{\GL}{GL}

\newtheorem{theorem}{Theorem}[section]
\newtheorem{theorem*}{Theorem}
\newtheorem{proposition}[theorem]{Proposition}
\newtheorem{corollary}[theorem]{Corollary}

\theoremstyle{definition}

\title{The third Milgram--Priddy class lifts}

\author{Markus Szymik}

\newdateformat{mydate}{\monthname~\twodigit{\THEYEAR}}
\date{\mydate\today}

\begin{document}

\maketitle








\section{Introduction}

Quillen has shown that the cohomology of the finite general linear groups with trivial coefficients in their defining characteristic vanishes in the stable range: when the cohomological degree is small with respect to the rank of the matrices~\cite{Quillen}. Specifically, it is known from Maazen's thesis~\cite{Maazen} that~$\rmH^d(\GL_{r}(\bbF_2);\bbF_2)$ is trivial in the range~\hbox{$0<d<r/2$}. Quillen's~(unpublished) stability arguments had to exclude the field~$\bbF_2$, as does the recent extension of his work by Sprehn and Wahl~\cite{Sprehn+Wahl}. Milgram--Priddy~\cite{Milgram+Priddy} and Lahtinen--Sprehn~\cite{Lahtinen+Sprehn} described non-zero classes far away from this bound. The question remained for a while: what happens at the~`edge'~\cite{Hepworth} of this region, for~$\rmH^n(\GL_{2n}(\bbF_2);\bbF_2)$? This has now been answered almost completely by Galatius, Kupers, and Randal-Williams~\cite[Thm.~B]{GKR-W}, who extended the stable range substantially, showing the vanishing of the relative homology~$\rmH_d(\GL_r(\bbF_2),\GL_{r-1}(\bbF_2);\bbF_2)=0$ for~\hbox{$d<2(r-1)/3$}. It follows that~$\rmH^n(\GL_{2n}(\bbF_2);\bbF_2)=0$ for all~$n\geqslant4$. In contrast, the non-triviality of~$\rmH^1(\GL_2(\bbF_2);\bbF_2)$ and~$\rmH^2(\GL_4(\bbF_2);\bbF_2)$ is well-known, because of the isomorphisms~$\GL_2(\bbF_2)\cong\rmS_3$ and~$\GL_4(\bbF_2)\cong\rmA_8$, whence the characteristic classes of the sign and the spin representations give non-zero representatives for them. The purpose of the present text is to resolve the remaining case. 

\begin{theorem}\label{thm:main}
We have~$\rmH^3(\GL_6(\bbF_2);\bbF_2)\not=0$.
\end{theorem}

Hepworth's work~\cite[Thm.~G]{Hepworth} now implies~(as Galatius, Kupers, and Randal-Williams~\cite[Lem.~6.7]{GKR-W} explain again) that Theorem~\ref{thm:main} automatically has the following stronger consequence:

\begin{corollary}
We have~$\dim\rmH_3(\GL_6(\bbF_2);\bbF_2)=1$, and the non-trivial element is the third power of the non-zero class in~$\rmH_1(\GL_2(\bbF_2);\bbF_2)$ under block-sum multiplication. 
\end{corollary}


As explained in~\cite[Sec.~6.3]{GKR-W} as well, Theorem~\ref{thm:main} also solves the remaining case of a problem posed by Priddy~\cite[Sec.~5]{Priddy}. A maximal elementary abelian $2$--subgroup~$M$~of the group $\GL_{2n}(\bbF_2)$ consists of the block matrices
\[
\begin{bmatrix}
\;1 & A\;\\
\;0 & 1\;
\end{bmatrix}
\]
with $A=(a_{ij})$ in the matrix ring~$\rmM_n(\bbF_2)\cong M$. Because the group is elementary abelian, we have an isomorphism~\hbox{$\rmH^*(M;\bbF_2)\cong\bbF_2[\,\alpha_{ij}\,]$} with the classes~$\alpha_{ij}$ in degree~$1$. The {\it Milgram--Priddy class}~(see~\cite{Milgram+Priddy}) is defined as~\hbox{$\det_n=\det(\alpha_{ij})$} and it lives in the invariants~$\rmH^n(M;\bbF_2)^W$, where~\hbox{$W\cong\GL_n(\bbF_2)\times\GL_n(\bbF_2)$} is the Weyl group of $M$ in~\hbox{$G=\GL_{2n}(\bbF_2)$}. Priddy posed the problem to decide whether or not we can lift this class along the restriction~\hbox{$\rmH^n(\GL_{2n}(\bbF_2);\bbF_2)\to\rmH^n(M;\bbF_2)^W$}.

\begin{corollary}
The third Milgram--Priddy class lifts.
\end{corollary}

It was known before that the first and second Milgram--Priddy classes lift, and it is a consequence of the work of Galatius, Kupers, and Randal-Williams that the third one is the last one that has the chance to do so.


One should think that a singular result like Theorem~\ref{thm:main} can be checked by computer, and this is true in theory, but currently only in theory. In practice, however, the computational complexity of this particular problem does not allow for a brute force approach: the group~$\GL_6(\bbF_2)$ has order about~\hbox{$\sim20\cdot 10^9$}. 

In the following Section~\ref{sec:strategy}, we explain a general strategy that we employ here to reduce our problem to computations that can routinely be done by machine. The computations themselves are explained in Section~\ref{sec:computations}. All cohomology will be with coefficients in the prime field~$\bbF_2$ unless otherwise indicated.


\section{The strategy}\label{sec:strategy}

In this section, we explain the general strategy that we exploit in the following section to show that certain cohomology groups~$\rmH^d(\GL_r(\bbF_2))$ are non-zero.

The general linear group~$\GL(V)$ of a vector space~$V$ acts on the chains of subspaces of~$V$, with the parabolic subgroups as the stabilizer subgroups of this action. The Borel subgroups are the stabilizers of the maximal chains, the flags. The isotropy spectral sequence of the action on the space of these chains, the Tits building, establishes a relation between the~(co)homological invariants of these groups. We will use the following result, which holds in arbitrary characteristic:

\begin{proposition}{\bf (Webb~\cite[Cor.~C]{Webb})}\label{prop:Webb}
Let~$G$ be a finite Chevalley group in defining characteristic~$p$, then
\[
\dim\rmH^d(G)=\sum_{P\geqslant B}(-1)^{\mathrm{rank}(P)}\dim\rmH^d(P)
\]
for all~$d>0$, where the sum is over the proper parabolic subgroups of~$G$ that contain a fixed Borel subgroup~$B$.
\end{proposition}

This result reduces the computation of the cohomology of~$G=\GL_r(\bbF_2)$ to the computation of the cohomology of its proper parabolic subgroups~$P$, which are much smaller, as we shall see.

Our aim is to show that certain cohomology groups~$\dim\rmH^d(\GL_r(\bbF_2))$ are non-zero. This follows when we can show that their dimension is odd.~(For $d=n$ and~\hbox{$r=2n$} we will then know that this dimension is~$1$, as explained in the introduction.) For that reason, we can ignore the signs in Proposition~\ref{prop:Webb}, even though they are easy to work out.

There is another simplification that we can do. The parabolic subgroups of a general linear group~$\GL(V)$ that contain a fixed Borel subgroup are indexed by the ordered partitions~$\lambda=(\lambda_1,\lambda_2,\dots,\lambda_k)$ of~\hbox{$\lambda_1+\lambda_2+\dots+\lambda_k=\dim(V)$}: the dimensions of the associated graded pieces of the chain that corresponds to the subgroup. We shall write~$\rmP(\lambda)$ for the parabolic subgroup of an ordered partition~$\lambda$. There is an obvious involution
\[
\lambda=(\lambda_1,\lambda_2,\dots,\lambda_k)\mapsto(\lambda_k,\lambda_{k-1},\dots,\lambda_1)=\lambda^\vee
\]
on the set of ordered partitions, and we have an isomorphism~$\rmP(\lambda^\vee)\cong\rmP(\lambda)$ of groups. Therefore, the ordered partitions~$\lambda$ with~$\lambda^\vee\not=\lambda$ come in pairs that do not contribute to the parity of~$\dim\rmH^d(\GL_r(\bbF_2))$. This shows that it is enough to compute~$\dim\rmH^d(\rmP(\lambda))$ for the {\it symmetric} ordered partitions~$\lambda$ of~$r$, those with~\hbox{$\lambda^\vee=\lambda$}, and there are substantially fewer of those. 

In summary, we have:

\begin{proposition}\label{prop:strategy}
\[
\dim\rmH^d(\GL_r(\bbF_2))\equiv\sum_{\lambda}\dim\rmH^d(\rmP(\lambda))\text{\upshape~mod }2,
\]
where the summation is over the proper~($\lambda\not=(r)$ that is) symmetric ordered partitions of~$r$. 
\end{proposition}

We see that the left hand side $\dim\rmH^d(\GL_r(\bbF_2))$ is non-zero once we have shown that the right hand side is odd.


\section{The computations}\label{sec:computations}

In this section, we explain the computations needed as an input for the arguments in the previous section to imply our main results.

\subsection{The second Milgram--Priddy class}\label{sec:2}

Let us warm up by applying the strategy outlined in the previous section to show that the second Milgram--Priddy class in~$\rmH^2(\GL_4(\bbF_2))$ is non-zero. Of course, there is an isomorphism~$\GL_4(\bbF_2)\cong\rmA_8$, and the non-triviality of its second mod~$2$ cohomology is well-known. The point here is to explain our strategy of computation in a known case.

For the group~$\GL_4(\bbF_2)$, there are only three relevant proper parabolic subgroups that contain the Borel subgroup, which is a~$2$--Sylow subgroup: these subgroups correspond to the symmetric ordered partitions~$(1,1,1,1)$,~$(1,2,1)$, and~$(2,2)$ of the~$4$--element set~$\{1,2,3,4\}$. The parabolic subgroup corresponding to the first one is the Borel--Sylow, the parabolic subgroup corresponding to the last one is the Milgram--Priddy subgroup.

We can compute the three relevant cohomology groups using the GAP package HAP~(see~\cite{GAP} and~\cite{HAP}).

\begin{proposition}
The second mod~$2$ cohomology of the proper parabolic subgroups of~$\GL_4(\bbF_2)$ is as follows.
\begin{center}
\begin{tabular}{|c|c|c|}\hline
\text{\upshape symmetric partition}~$\lambda$ & \text{\upshape order of the group}~$\rmP(\lambda)$ &~$\dim\rmH^2(\rmP(\lambda);\bbF_2)$\\ \hline
$(1,1,1,1)$ &~$64=2^6$ &~$7$ \\ 
$(1,2,1)$ &~$192=2^6\cdot 3$ &~$4$ \\ 
$(2,2)$ &~$576=2^6\cdot 9$ &~$4$ \\ \hline
\end{tabular}
\end{center}
\end{proposition}

Note that all of these subgroup are considerable smaller than the full general linear group~$\GL_4(\bbF_2)$, which is of order~$20.160=2^6\cdot 315$.
 

We check that the sum~$\pm7\pm4\pm4$ is odd. As explained in the previous section, this implies that~$\rmH^2(\GL_4(\bbF_2))$ is non-zero.


\subsection{The third Milgram--Priddy class}\label{sec:3}

We now complete the proof of our main result, Theorem~\ref{thm:main}. 

In order to show that the third Milgram--Priddy class is non-zero, we follow the same workflow as in Section~\ref{sec:2}. Only the computational complexity increases because of the size of the group~$\GL_6(\bbF_2)$.

Of the 32 ordered partitions of the~$6$--element set~$\{1,2,3,4,5,6\}$, only 8 are symmetric, namely $( 1, 1, 1, 1, 1, 1 )$, $( 1, 1, 2, 1, 1 )$, $( 1, 2, 2, 1 )$, $( 1, 4, 1 )$, $( 2, 1, 1, 2 )$, $( 2, 2, 2 )$, $( 3, 3 )$, and~$( 6 )$.
The parabolic subgroup corresponding to the first one is the Borel, which is also a~$2$--Sylow subgroup of~$\GL_6(\bbF_2)$. 
The parabolic subgroup corresponding to the last one is~$\GL_6(\bbF_2)$ itself and not proper. 


For our strategy to work, we need to know the third mod~$2$ cohomology of the~7 proper parabolic subgroups, and these are small enough so that they can be dealt with by a machine. Again, we used the GAP package HAP for that~(see~\cite{GAP} and~\cite{HAP}). The result is as follows.

\begin{proposition}\label{prop:computations}
The third mod~$2$ cohomology of the proper parabolic subgroups of~$\GL_6(\bbF_2)$ is as follows.
\begin{center}
\begin{tabular}{|c|c|c|}\hline
\text{\upshape symmetric partition}~$\lambda$ & \text{\upshape order of the group}~$\rmP(\lambda)$ &~$\dim\rmH^3(\rmP(\lambda);\bbF_2)$\\ \hline
$(1,1,1,1,1,1)$ &~$32.768 = 2^{15}$ &~$47$ \\ 
$(1, 1, 2, 1, 1)$ &~$98.304 = 2^{15}\cdot 3$ &~$28$ \\
$(1, 2, 2, 1)$ &~$294.912 = 2^{15}\cdot 9$ &~$16$\\
$(1, 4, 1)$ &~$10.321.920 = 2^{15}\cdot 315$ &~$5$ \\
$(2, 1, 1, 2)$ &~$294.912 = 2^{15}\cdot 9$ &~$24$ \\
$(2, 2, 2)$ &~$884.736 = 2^{15}\cdot 27$ &~$17$ \\
$(3, 3)$ &~$14.450.688 = 2^{15}\cdot 441$ &~$6$ \\ \hline
\end{tabular}
\end{center}
\end{proposition}

The table also displays the size of the proper parabolic subgroups, and we see that they are substantially smaller than the group~$\GL_6(\bbF_2)$, which is of order~$20.158.709.760= 2^{15}\cdot 615.195$.


\begin{proof}[Proof of Theorem~\ref{thm:main}]
The sum~$\pm47\pm28\pm16\pm5\pm24\pm17\pm6$ is odd.
Propositions~\ref{prop:strategy} and~\ref{prop:computations} thus imply that~$\dim\rmH^3(\GL_6(\bbF_2))$ has to be odd as well. Theorem~\ref{thm:main} follows.
\end{proof}

A recent paper~\cite{J-FT} by Johnson-Freyd and Treumann contains information about the third homology of some {\it sporadic} finite (simple) groups. 


\section*{Acknowledgements}

The computations reported on here were done in the Summer of 2016; the author apologies for the delay in making them public and thanks for encouragements to do so.



\vfill

\parbox{\linewidth}{%
Markus Szymik\\
Department of Mathematical Sciences\\
NTNU Norwegian University of Science and Technology\\
7491 Trondheim\\
NORWAY\\
\href{mailto:markus.szymik@ntnu.no}{markus.szymik@ntnu.no}\\
\href{https://folk.ntnu.no/markussz}{folk.ntnu.no/markussz}}

\end{document}